\documentclass[12pt, twoside, leqno]{article}

\usepackage{amsmath,amsthm}
\usepackage{amssymb}
\usepackage{tikz}
\usepackage{enumerate}

\usepackage{graphicx}

\newtheorem{thm}{Theorem}[section]
\newtheorem{cor}[thm]{Corollary}

\theoremstyle{definition}
\newtheorem{defin}[thm]{Definition}

\numberwithin{equation}{section}

\begin{document}


\baselineskip=17pt


\title{On the Axiomatic Systems of Singular Cohomology Theory}

\author{Anzor Beridze\\
Department of Mathematics\\
Batumi Shota Rustaveli State University\\
35, Ninoshvili St., Batumi 6010, Georgia\\ 
E-mail: a.beridze@bsu.edu.ge
\and 
Leonard Mdzinarishvili\\
Department of Mathematics\\
Georgian Technical University\\
77, Kostava St., Tbilisi 0171, Geogia\\
E-mail: l.mdzinarishvili@gtu.ge}

\date{}

\maketitle


\renewcommand{\thefootnote}{}



\renewcommand{\thefootnote}{\arabic{footnote}}
\setcounter{footnote}{0}


\begin{abstract}
On the category of pairs of topological spaces having a homotopy type of $CW$ complexes the singular (co)homology theory was axiomatically studied by J.Milnor \cite{8}. In particular, Milnor gave {\it additivity axiom} for a (co)homology theory and proved that any additive (co)homology theory on the given category is isomorphic to the singular (co)homology. On the other hand,  the singular homology is a homology with compact support \cite{3}. In the paper \cite{6},  L. Mdzinarishvili proposed {\it partially compact support property} for a cohomology theory and gave another axiomatic characterization of the singular cohomology theory  \cite{6}. In this paper, we will give  additional different axiomatic characterizations of the singular cohomology theory. Moreover, we will study connections of the mentioned axiomatic systems. 

\end{abstract}

\section{Introduction}

Let $ \mathcal{K}^2_{Top}$ be the category of pairs of topological spaces,  $\mathcal{K}^2$ be any admissible subcategory \cite{3} of the category $\mathcal{K}^2_{Top}$ and $ {\mathcal Ab}$ be the category of abelian groups.

A sequence $H^*=\{H^n, \delta \}_{n \in \mathbb{Z}}$ of contravariant functors $ H^n: \mathcal{K}^2 \to Ab$ is called a cohomological \cite{3} if:

$1_H$) for each object $(X,A) \in {\mathcal{K} ^2}$, $G \in {\mathcal Ab}$ and $n \in \mathbb{Z}$ there exists a $\delta$-homomorphism
\begin{equation}
\delta :H^{n-1}(A;G) \rightarrow H^n(X,A;G),
\end{equation}
where $H^{n-1}(A;G)\equiv H^{n-1}(A, \emptyset ;G)$.

$2_H$) the diagram  
\begin{equation}
\begin{tikzpicture}

\node (A) {${H}^{n-1}(B;G)$};
\node (B) [node distance=4cm, right of=A] {$H^{n}(Y,B;G)$};

\node (A1) [node distance=2cm, below of=A] {${H}^{n-1}(A;G)$};
\node (B1) [node distance=2cm, below of=B] {$H^{n}(X,A;G)$};

\draw[->] (A) to node [above]{$\delta$}(B);

\draw[->] (A1) to node [above]{$\delta$}(B1);

\draw[->] (A) to node [right]{$(f_{|A})^*$}(A1);

\draw[->] (B) to node [right]{$f^*$}(B1);

\end{tikzpicture}
\end{equation}
is commutative for each continuous mapping  $f: (X,A) \to (Y,B)$ ($f^* : H^n(Y,B;G) \to H^n(X,A;G)$ and $(f_{|A})^*  : H^{n-1}(B;G) \to H^{n-1}(A;G)$ are  the homomorphisms induced by $f: (X,A) \to (Y,B)$ and $f_{|A}:A \to B$, correspondingly).

A cohomological sequence $H^*=\{ H^n , \delta \}_{n \in \mathbb{Z}}$ is called the  cohomology theory in the Eilenberg-Steenrod sense on the category  $\mathcal{K}^2$ if it satisfies the homotopy, excision, exactness and dimension axioms for the given $\mathcal{K}^2$ category \cite{3}. It is known that up to an isomorphism such a cohomology theory is unique on the category $\mathcal{K}^2_{\mathcal{P}ol_C}$ of pairs of compact polyhedrons \cite{3}. Analogously, is defined  the unique homology theory $H_*=\{ H_n , \partial \}_{n \in \mathbb{Z}} $ on the category $\mathcal{K}^2_{\mathcal{P}ol_C}$ \cite{3}.

Let $\mathcal{K}^2_{\mathcal{CW}}$ be the category of pairs of topological spaces having a homotopy type of $CW$ complexes. The singular (co)homology theory on the category $\mathcal{K}^2_{\mathcal{CW}}$ was first axiomatically described by J. Milnor \cite{8}. He proved the uniqueness theorem using the Eilenberg-Steenrod axioms and {\it the  additivity axiom}:

AD ({\it Additivity axiom}): If $X$ is the disjoint union of open subspaces $X_{\alpha}, ~ \alpha \in \mathcal{A}$ with inclusion maps $i_{\alpha}:X_{\alpha} \to X$, all belonging to the category $\mathcal{K}^2_{\mathcal{CW}}$, then the homomorphisms $i_{\alpha  }^*:H^n(X;G) \to H^n(X_{\alpha};G)$ ($i_{\alpha *}:H_n(X_{\alpha};G) \to H_n(X;G)$) must provide a projective (an injective) representation of $H^n(X;G)$ ($H_n(X;G)$) as a direct product (a direct sum).

In \cite{8} the following is proved:
\begin{thm} (see the Uniqueness Theorem in \cite{8}) Let $H_*$ be an additive homology theory on the category $\mathcal{K}^2_{\mathcal{CW}}$ with coefficients group $G$. Then for each $(X,A)$ in $\mathcal{K}^2_{\mathcal{CW}}$ there is a natural isomorphism between $H_n(X,A;G)$ and the  n-th singular homology group  of $(X,A)$ with coefficients in $G$.
\end{thm}

At the end of the paper \cite{8} it is mentioned that {\it "the corresponding theorem for cohomology groups can be proved in the same way"}, which can be formulated in the following way:
\begin{thm} Let $H^*$ be an additive cohomology theory on the category $\mathcal{K}^2_{\mathcal{CW}}$ with coefficients group $G$. Then for each $(X,A)$ in $\mathcal{K}^2_{\mathcal{CW}}$ there is a natural isomorphism between $H^n(X,A;G)$ and the n-th singular cohomology group of $(X,A)$ with coefficients in $G$.
\end{thm}

Another axiomatic characterization of the singular cohomology theory was given by L. Mdzinarishvili. In particular, in  \cite{6} a cohomology theory with a {\it partially compact support} was defined:

\begin{defin}({\it see Definition 2 in \cite{6}}) A cohomology theory $H^*$ for which there is a finite exact sequence
\begin{equation}
0  \to {\varprojlim}^{2n-3} H^1 F_{\alpha} \to \dots \to {\varprojlim}^{1} H^{n-1} F_{\alpha} \to H^n X \to
\end{equation}
$$\to {\varprojlim} H^n F_{\alpha} \to {\varprojlim}^{2} H^{n-1} F_{\alpha} \to \dots \to {\varprojlim}^{2n-2} H^1 F_{\alpha} \to 0,$$
where  $\mathbf{F}=\{ F_\alpha \}_{\alpha \in \mathcal{A}}$ is the  direct system of all compact subspaces $F_{\alpha}$ of $X$  directed by the inclusion, is called a cohomology theory with partially compact supports.
\end{defin}

Using the partially compact support property, the following results are obtained  \cite{6}:

\begin{thm}({\it see Theorem 4 in \cite{6}}) For the singular cohomology of any topological space $X$ there is a finite exact sequence 
\begin{equation}
0  \to {\varprojlim}^{2n-3} H^1_s F_{\alpha} \to \dots \to {\varprojlim}^{1} H^{n-1}_s F_{\alpha} \to H^n_s X \to
\end{equation}
$$\to {\varprojlim} H^n_s F_{\alpha} \to {\varprojlim}^{2} H^{n-1}_s F_{\alpha} \to \dots \to {\varprojlim}^{2n-2} H^1_s F_{\alpha} \to 0,$$
where $H^*_s F_{\alpha}=H^*_s (F_{\alpha}; G)$,  $H^*_s X=H^*_s (X; G).$
\end{thm}

\begin{cor}({\it see Corollary 1 in \cite{6}}) If $X$ is a polyhedron and $\{ F_{\alpha} \}$ is a system of compact subspaces of $X$, then for the singular cohomology there is a short exact sequence 
\begin{equation}
0  \to {\varprojlim}^{1} H^{n-1}_s F_{\alpha} \to H^n_s X \to {\varprojlim} H^n_s F_{\alpha} \to 0.
\end{equation}

\end{cor}

\begin{thm} ({\it see Theorem 5 in \cite{6}}) Let $h$ be a homomorphism from cohomology $H$ to cohomology $H'$, that is an isomorphism for one-point spaces. If $H$ and $H'$ have partially compact supports, then $h$ is an isomorphism for any polyhedron pair.
\end{thm}

In the paper we will propose  the universal coefficients formula as one more axiomatic characterization of the singular cohomology theory (cf. Axiom D in \cite{1}). 

UCF ({\it Universal coefficients formula}): For each space $X$ there exists a functorial  exact sequence:

\begin{equation}
0 \to Ext(H^s_{n-1}(X),G) \to H^n (X;G) \to Hom(H^s_n(X),G) \to 0,
\end{equation}
where $H_*^s(X)$ is the singular homology groups with coefficients in $\mathbb{Z}$.

Besides, we consider a cohomological  exact bifunctor which has a compact support property for an injective coefficients group and  prove the uniqueness  theorem for it (cf. Theorem 1 in \cite{9}).

EFSA ({\it Exact functor of the second argument}): A cohomological sequence $H^*=\{ H^n, \delta \}$ is called exact functor  of the second argument if for each short exact sequence of abelian groups
\begin{equation}
0 \to G \to G' \to G'' \to 0,
\end{equation} 
and for each space $X \in   \mathcal{K}^2$ there is a functorial natural long exact sequence:
\begin{equation}
\dots \to {H}^{n-1}(X;G'') \to {H}^n(X;G) \to {H}^n(X;G') \to {H}^n(X;G'') \to \dots ~.
\end{equation}

CSI ({\it Compact support for an injective coefficients group})  For each injective coefficients group $G$ there is an isomorphism
\begin{equation}
H^n(X;G) \approx \varprojlim H^n(F_\beta ;G),
\end{equation}
where $\mathbf{F}=\{ F_\alpha \}_{\alpha \in \mathcal{A}}$ is a direct system of compact subspaces of $X$.

Moreover, we study the connections between these different axiomatic approaches.

\section{A Nontrivial internal cohomological extension}

Let  $H^*=\{H^n, \delta \}_{n \in \mathbb{Z} } $ be a cohomological sequence defined on the category $\mathcal{K}^2$ of pairs of some topological spaces, which contains the category $\mathcal{K}^2_{\mathcal{P}ol_C}$ of pairs of compact polyhedrons and let $h^*$ be a cohomology theory in the Eilenberg-Steenrod sense on the subcategory $\mathcal{K}^2_{\mathcal{P}ol_C}$.   The sequence $H^*$ is called  an extension of the cohomology theory  $h^*$ in the Eilenberg-Steenrod sense defined on the category $\mathcal{K}^2_{\mathcal{P}ol_C}$  to the category $\mathcal{K}^2$ if   $H^*_{|\mathcal{P}ol_C}=h^* $ \cite{7}. The cohomological sequence  $H^* $  is called a nontrivial internal extension of the cohomology theory  $h^*$ if the following conditions are fulfilled:

$1_{NT}$)  $H^*$ is an extension of the cohomology theory  $h^*$ in the Eilenberg-Steenrod sense defined on the category $\mathcal{K}^2_{\mathcal{P}ol_C}$ to the category $\mathcal{K}^2$;

$2_{NT}$) the exact sequence 
\begin{equation}
0 \to {\varprojlim}^1 H^{n-1}(F_\alpha,E_\alpha;G) \to H^n(X,A;G) \to {\varprojlim} H^{n}(F_\alpha, E_\alpha;G) \to 0
\end{equation}
holds, where $\mathbf{F}= \{ (F_\alpha,E_\alpha) \}$ is a direct system of pairs of compact subspaces of $X$, such that $E_\alpha =F_\alpha \cap A$;

$3_{NT}$) the commutative diagram  
\begin{equation}
\begin{tikzpicture}

\node (A) {${\varprojlim}^1 {H}^{n-1}(L_\beta,M_\beta;G)$};
\node (B) [node distance=4cm, right of=A] {$H^{n}(Y,B;G)$};

\node (A1) [node distance=2cm, below of=A] {${\varprojlim}^1 {H}^{n-1}(F_\alpha,E_\alpha;G)$};
\node (B1) [node distance=2cm, below of=B] {$H^{n}(X,A;G)$};

\draw[->] (A) to node [above]{}(B);

\draw[->] (A1) to node [above]{}(B1);

\draw[->] (A) to node [right]{${\varprojlim}^1  \tilde{f}^*$}(A1);

\draw[->] (B) to node [right]{$f^*$}(B1);

\end{tikzpicture}
\end{equation}holds for any continuous mapping $f:(X,A) \to (Y,B)$ from $\mathcal{K}^2$, where $\tilde{f}^* : \{ H^{n}(L_\beta,M_\beta;G),(j_{\beta,\beta'})^*, \mathcal{B}  \} \to \{ H^{n}(F_\alpha,E_\alpha), (i_{\alpha,\alpha'})^*, \mathcal{A}  \}$ is mapping of the inverse systems induced by $f$;

$4_{NT}$) $H^*$ satisfies the exactness axiom.
\begin{thm} If $H^*$ is a nontrivial internal extension of cohomology theory  $h^*$ in the Eilenberg-Steenrod sense defined on the category $\mathcal{K}^2_{\mathcal{P}ol_C}$ to the category $\mathcal{K}^2_{\mathcal{CW}}$, then it is a theory in the Eilenberg-Steenrod sense on the category $\mathcal{K}^2_{\mathcal{CW}}$.  
\end{thm} 
 Note that in the next section a nontrivial internal extension of cohomology theory will be called a cohomology of Mdzinarishvili sense as well. 

\begin{proof} Homotopy Axiom. By the property $4_{NT}),$ it is sufficient to show for the absolute case. For this aim, let us show that for each $X$ from the category $\mathcal{K}^2_\mathcal{CW}$ the inclusions $i_0, i_1 : X \rightarrow X \times I$ induce the same homomorphism $i_0^*=i_1^*:H^n(X \times I;G) \rightarrow H^n(X;G)$, where $i_0(x)=(x,0), ~~ i_1(x)=(x,1)$. Let $\mathbf{F}=\{ F_\alpha \}_{\alpha \in \mathcal{A}}$ be a direct system of compact subspaces $F_\alpha \subset X$. In this case the system $\mathbf{F }\times I=\{ F_\alpha \times I\}_{\alpha \in \mathcal{A}}$ is a confinal subsystem of the system $\mathbf{E }=\{ E_\beta \}_{\beta \in \mathcal{B}}$ of all compact subspaces $E_\beta \subset X \times I$. Indeed, for each $E_\beta$ consider its projection $F_\alpha = p (E_\beta)$, where $p:X \times I \rightarrow X$ is defined by the formula $p(x,t)=x$. In this case $F_\alpha $ is compact and  $E_\beta \subset F_\alpha \times I$. Let $\tilde{i_0}=\{i_0^\alpha\}, \tilde{i_1}=\{i_1^\alpha\}: \{ F_\alpha \}_{\alpha \in \mathcal{A}} \rightarrow \{ F_\alpha \times I\}_{\alpha \in \mathcal{A}}$ are the canonical mappings induced by $i_0$ and $i_1$. Consider the corresponding mappings $\tilde{i_0}^*=\{(i_0^\alpha)^*\}, \tilde{i_1}^*=\{(i_1^\alpha)^*\}: \{ H^n(F_\alpha \times I;G) \}_{\alpha \in \mathcal{A}} \rightarrow \{ H^n(F_\alpha ;G)\}_{\alpha \in \mathcal{A}}$ and the induced commutative diagram:
\begin{equation}
\begin{tikzpicture}
\node (A) {0};
\node (B) [node distance=2.5cm, right of=A] {${\varprojlim}^1 {H}^{n-1}(F_{\alpha} \times I;G)$};
\node (C) [node distance=4cm, right of=B] {$H^n(X \times I;G)$};
\node (D) [node distance=4cm, right of=C] {${\varprojlim} {H}^{n}(F_{\alpha} \times I;G)$};
\node (E) [node distance=2.5cm, right of=D] {0};

\node (A1) [node distance=2cm, below of=A] {0};
\node (B1) [node distance=2cm, below of=B] {${\varprojlim}^1 {H}^{n-1}(F_{\alpha};G)$};
\node (C1) [node distance=2cm, below of=C] {$H^n(X;G)$};
\node (D1) [node distance=2cm, below of=D] {${\varprojlim} {H}^{n}(F_{\alpha};G)$};
\node (E1) [node distance=2cm, below of=E] {0,};

\draw[->] (A) to node [above]{}(B);
\draw[->] (B) to node [above]{}(C);
\draw[->] (C) to node [above]{}(D);
\draw[->] (D) to node [above]{}(E);

\draw[->] (A1) to node [above]{}(B1);
\draw[->] (B1) to node [above]{}(C1);
\draw[->] (C1) to node [above]{}(D1);
\draw[->] (D1) to node [above]{}(E1);

\draw[->] (B) to node [right]{${\varprojlim}^1 ( {i_k^{\alpha}})^*$}(B1);
\draw[->] (C) to node [right]{$i_k^*$}(C1);
\draw[->] (D) to node [right]{${\varprojlim} ( {i_k^\alpha})^*$}(D1);
\end{tikzpicture}
\end{equation}
where $k=0,1$. On the other hand, for each $F_\alpha$ the inclusion maps $i_0^\alpha, i_1^\alpha :F_\alpha \rightarrow F_\alpha \times I$ belongs to the category $\mathcal{K}^2_{{\mathcal{P}ol_C}}$ and are homotopic and so by virtue of condition $1_{NT})$, $(i_0^\alpha)^*=(i_1^\alpha)^* :H^n(F_\alpha \times I;G) \rightarrow H^n (F_\alpha ;G)$. Therefore, ${\varprojlim} ( {i_0^\alpha})^*={\varprojlim} ( {i_1^\alpha})^*$ and ${\varprojlim}^1 ( {i_0^{\alpha}})^*={\varprojlim}^1 ( {i_1^{\alpha}})^*$. 
On the other hand, if we consider the system $\tilde{p}=\{p_\alpha\}:\{F_\alpha \times I \} \rightarrow \{F_\alpha \} $ induced by the projection $p:X \times I \rightarrow X$, then it induces the following commutative diagram:
\begin{equation}
\begin{tikzpicture}
\node (A) {0};
\node (B) [node distance=2.5cm, right of=A] {${\varprojlim}^1 {H}^{n-1}(F_{\alpha};G)$};
\node (C) [node distance=4cm, right of=B] {$H^n(X ;G)$};
\node (D) [node distance=4cm, right of=C] {${\varprojlim} {H}^{n}(F_{\alpha} ;G)$};
\node (E) [node distance=2.5cm, right of=D] {0};

\node (A1) [node distance=2cm, below of=A] {0};
\node (B1) [node distance=2cm, below of=B] {${\varprojlim}^1 {H}^{n-1}(F_{\alpha}\times I;G)$};
\node (C1) [node distance=2cm, below of=C] {$H^n(X\times I;G)$};
\node (D1) [node distance=2cm, below of=D] {${\varprojlim} {H}^{n}(F_{\alpha}\times I;G)$};
\node (E1) [node distance=2cm, below of=E] {0.};

\draw[->] (A) to node [above]{}(B);
\draw[->] (B) to node [above]{}(C);
\draw[->] (C) to node [above]{}(D);
\draw[->] (D) to node [above]{}(E);

\draw[->] (A1) to node [above]{}(B1);
\draw[->] (B1) to node [above]{}(C1);
\draw[->] (C1) to node [above]{}(D1);
\draw[->] (D1) to node [above]{}(E1);

\draw[->] (B) to node [right]{${\varprojlim}^1  {p_{\alpha}}^*$}(B1);
\draw[->] (C) to node [right]{$p^*$}(C1);
\draw[->] (D) to node [right]{${\varprojlim}  {p_\alpha}^*$}(D1);
\end{tikzpicture}
\end{equation}
Note that, for each $\alpha \in \mathcal{A}$, $i_k  \circ p_\alpha \sim 1_{F_\alpha \times I}$ in the category $\mathcal{K}^2_{{\mathcal{P}ol_C}}$ and so $p_\alpha^* \circ i_k^*=(1_{F_\alpha \times I})^*=1_{H^n(F_\alpha \times I; G)}, ~k=0,1.$ Therefore, we obtain:
\begin{equation}
\varprojlim p_\alpha^* \circ \varprojlim (i_k)^* = \varprojlim (1_{F_\alpha \times I})^*,
\end{equation}
\begin{equation}
{\varprojlim}^1 p_\alpha^* \circ {\varprojlim}^1 (i_k)^* = {\varprojlim}^1 (1_{F_\alpha \times I})^*.
\end{equation}
Consequently, by the diagrams (2.3) and (2.4) we obtain that $p^* \circ i_k^* ,~k=0,1$ are isomorphisms. On the other hand, $p \circ i_0 =p \circ i_1$ and so $i_0^* \circ p^* =i_1^* \circ p^*$. Therefore
\begin{equation}
i_0^*  \circ (p^* \circ i_0^*) = (i_0^*  \circ p^*) \circ i_0^*=(i_1^*  \circ p^*) \circ i_0^*=i_1^*  \circ (p^* \circ i_0^*).
\end{equation}
Consequently, we obtain that $i_0^*=i_1^*.$

Excision Axiom. Let $(X,A) \in \mathcal{K}^2_{\mathcal{CW}}$, then $A$ is closed subspace of $X$. Let $U$ be an open subspace of $A$ such that $\bar{U} \subset Int A$. Consider the coresponding inclusion map $i_U : (X\setminus U, A\setminus U) \rightarrow (X,A)$ and  let us show that it induces the isomorphism:
$$ i_U^*: H^n (X,A;G) \cong H^n (X\setminus U, A  \setminus U;G).$$
Indeed, let $\mathbf{F}=\{ (F_\alpha,E_\alpha) \}_{\alpha \in \mathcal{A}}$ is the system of all compact pairs of subspaces of $X$, such that $E_\alpha =F_\alpha \cap A$. Let $\mathbf{L}=\{ (L_\beta,M_\beta) \}_{\beta \in \mathcal{B}}$ be the system of all compact pairs of subspaces of $X \setminus U$, such that $M_\beta = L_\beta \cap (A \setminus U)$. Let us show that $(i_U)^{-1}(\mathbf{F})=\{ ((i_U)^{-1}(F_\alpha),(i_U)^{-1}(E_\alpha)) \}_{\alpha \in \mathcal{A}}=\{ (F_\alpha \setminus U, E_\alpha \setminus U) \}_{\alpha \in \mathcal{A}}$ is a confinal subsystem of the system $\mathbf{L}=\{ (L_\beta,M_\beta) \}_{\beta \in \mathcal{B}}$. Indeed, let $(L_\beta,M_\beta)$ be any pair of the system $\mathbf{L}$. In this case, $L_\beta \subset X \setminus U \subset X$ and  $M_\beta = L_\beta
 \cap A$. Therefore, $(L_\beta,M_\beta)$ is a pair of the system $\mathbf{F}$. On the other hand, $(L_\beta,M_\beta) = (i_U)^{-1} (L_\beta,M_\beta)$. Consider the canonical mapping $\tilde{i_U}=\{i_U^\alpha \}: \{ (F_\alpha \setminus U, E_\alpha \setminus U) \} \rightarrow \{ (F_\alpha, E_\alpha ) \}$, which induces the mapping 
$$\tilde{i_U}^*=\{(i_U^\alpha)^* \}: \{ H^n(F_\alpha , E_\alpha  ;G) \} \rightarrow \{ H^n(F_\alpha \setminus U, E_\alpha \setminus U ;G) \}.$$
Consider the corresponding commutative diagram:
\begin{equation}
\begin{tikzpicture}
\node (A) {0};
\node (B) [node distance=3.5cm, right of=A] {${\varprojlim}^1 {H}^{n-1}(F_{\alpha}, E_\alpha ;G)$};
\node (C) [node distance=5.2cm, right of=B] {$H^n(X,A;G)$};
\node (D) [node distance=3cm, right of=C] {};

\node (A1) [node distance=2cm, below of=A] {0};
\node (B1) [node distance=2cm, below of=B] {${\varprojlim}^1 {H}^{n-1}(F_{\alpha} \setminus U, E_\alpha \setminus U;G)$};
\node (C1) [node distance=2cm, below of=C] {$H^n(X \setminus U, A \setminus U;G)$};
\node (D1) [node distance=2cm, below of=D] {};

\node (D2) [node distance=2cm, below of=C1] {${\varprojlim} {H}^{n}(F_{\alpha}, E_\alpha ;G)$};
\node (C2) [node distance=3cm, left of=D2] {};
\node (E2) [node distance=3.5cm, right of=D2] {0};
\draw[->] (D2) to node [above]{}(E2);
\draw[->] (C2) to node [above]{}(D2);

\node (C3) [node distance=2cm, below of=C2] {};
\node (D3) [node distance=2cm, below of=D2] {${\varprojlim} {H}^{n}(F_{\alpha} \setminus U, E_\alpha  \setminus U;G)$};
\node (E3) [node distance=3.5cm, right of=D3] {0.};
\draw[->] (D3) to node [above]{}(E3);
\draw[->] (C3) to node [above]{}(D3);

\draw[->] (A) to node [above]{}(B);
\draw[->] (B) to node [above]{}(C);
\draw[->] (C) to node [above]{}(D);

\draw[->] (A1) to node [above]{}(B1);
\draw[->] (B1) to node [above]{}(C1);
\draw[->] (C1) to node [above]{}(D1);

\draw[->] (B) to node [right]{${\varprojlim}^1 ( {i_U^{\alpha}})^*$}(B1);
\draw[->] (C) to node [right]{$i_U^*$}(C1);
\draw[->] (D2) to node [right]{${\varprojlim} ( {i_U^\alpha})^*$}(D3);
\end{tikzpicture}
\end{equation}
For each $\alpha \in \mathcal{A}$ consider $U_\alpha = F_\alpha \cap U$. In this case, $U_\alpha$ is an open subspace of $X \setminus U$ and $\bar{U}_\alpha \subset int E_\alpha$.  On the other hand,  $(F_\alpha \setminus U, E_\alpha \setminus U)=(F_\alpha \setminus U_\alpha, E_\alpha \setminus U_\alpha)$  and so $i_U^\alpha : (F_\alpha \setminus U_\alpha, E_\alpha \setminus U_\alpha) \rightarrow (F_\alpha , E_\alpha )$ is an inclusion map which belongs to the category $\mathcal{K}^2_{{\mathcal{P}ol_C}}$. Therefore, $i_U^\alpha$ induces the isomorphism  $(i_U^\alpha)^* : H^n(F_\alpha, E_\alpha;G) \cong H^n(F_\alpha \setminus U , E_\alpha \setminus U;G)$. Consequently, the homomorphisms  ${\varprojlim} ( {i_U^\alpha})^*$, ${\varprojlim}^1 ( {i_U^{\alpha}})^*$ and so $i_U^*$ are isomorphisms.

Dimension Axiom. For each one-point space $X=\{*\}$ the direct system $\mathbf{F}=\{F_\alpha\}_{\alpha \in \mathcal{A}}$ of compact subspaces is a constant direct system, where $F_\alpha =\{*\}$. Consequently, we obtain the following short exact sequence:
\begin{equation}
0 \to {\varprojlim}^1 H^{n-1}(*;G) \to H^n(*;G) \to {\varprojlim} H^{n}(*;G) \to 0.
\end{equation}
On the other hand, ${\varprojlim}^1 H^{n-1}(*;G)=0$ for all $n \in \mathbb{Z}$ and $H^{n}(*;G)={\varprojlim} H^{n}(*;G)=0$ for $n \neq 0$.

\end{proof}

\section{Relations between different axiomatic systems   for the singular cohomology theory}

In this paper we will say that a cohomological sequence $H^*$ defined on the category $\mathcal{K}^2_{\mathcal{CW}}$  of pairs of topological spaces having a homotopy type of $CW$ complexes is:

1)  a cohomology theory in the Milnor sense if it is  a cohomology theory in the Eilenberg-Steenrod sense and it satisfies an additivity axiom;

2) a cohomology theory in the Mdzinarishvili sense if it is a nontrivial internal extension of the cohomology theory in the Eilenberg-Steenrod sense defined on the category $\mathcal{K}^2_{\mathcal{P}ol_C}$ to the category $\mathcal{K}^2_{\mathcal{CW}}$;

3)  a cohomology theory in the Berikashvili-Mdzinarishvili-Beridze sense if it is an extension of the cohomology theory in the Eilenberg-Steenrod sense defined on the category $\mathcal{K}^2_{\mathcal{P}ol_C}$ to the category $\mathcal{K}^2_{\mathcal{CW}}$ and it satisfies the exactness and UCF axioms on the category $\mathcal{K}^2_{\mathcal{CW}}$ (cf. axiom D of \cite{1});

4) a cohomology theory in the Inasaridze-Mdzinarishvili-Beridze sense if it is an extension of the cohomology theory in the Eilenberg-Steenrod sense defined on the category $\mathcal{K}^2_{\mathcal{P}ol_C}$ to the category $\mathcal{K}^2_{\mathcal{CW}}$ and  it satisfies the EFSA and CSI axioms on the category  $\mathcal{K}^2_{\mathcal{CW}}$ (cf. \cite{9}).

\begin{thm} If $H^*$ is a cohomology  theory in the Mdzinarishvili sense defined on the category $\mathcal{K}^2_{\mathcal{CW}}$ of pairs of topological spaces having a homotopy type of $CW$ complexes,  then it is a cohomology theory in the Milnor sense.
\end{thm} 
\begin{proof}  By theorem 2.1, if $H^*$ is a cohomology theory in the Mdzinarishvili sense defined on the category 	$\mathcal{K}^2_{\mathcal{CW}}$, then it is a theory in the Eilenberg-Steenrod sense. Therefore, it is sufficient to  show that it is an additive theory. 

Let $X \in \mathcal{K}^2_{\mathcal{CW}}$ be a disjoint union of open subspaces $X_{\alpha}, ~ \alpha \in \mathcal{A}$ with inclusion maps $i_{\alpha}:X_{\alpha} \to X$, all belonging to the category $\mathcal{K}^2_{\mathcal{CW}}$. Let $\mathbf{F}=\{ F_\beta \}  _{ \beta \in \mathcal{B} }$ be a direct system of all compact subspaces of $X$. Let $ F_{\alpha , \beta }=X_\alpha \bigcap F_\beta $, then $\mathbf{F_\alpha}=\{ F_{\alpha ,\beta} \}  _{ \beta \in \mathcal{B} }$ is the subsystem of the system $\mathbf{F}$ and there is a natural inclusion $\mathbf {i_{\alpha  }}=\{ i_{\alpha , \beta} \} : \mathbf{ F_\alpha} \to \mathbf{F}$, which is induced by the inclusion $i_\alpha : X_\alpha \to X.$ Consequently, for each $\alpha \in \mathcal{A}$ we have the following commutative diagram:
\begin{equation}
\begin{tikzpicture}
\node (A) {0};
\node (B) [node distance=2.5cm, right of=A] {${\varprojlim}^1 {H}^{n-1}(F_{\beta};G)$};
\node (C) [node distance=3.3cm, right of=B] {$H^n(X;G)$};
\node (D) [node distance=3.3cm, right of=C] {${\varprojlim} {H}^{n}(F_{\beta};G)$};
\node (E) [node distance=2.5cm, right of=D] {0};

\node (A1) [node distance=2cm, below of=A] {0};
\node (B1) [node distance=2cm, below of=B] {${\varprojlim}^1 {H}^{n-1}(F_{\alpha , \beta};G)$};
\node (C1) [node distance=2cm, below of=C] {$H^n(X_\alpha;G)$};
\node (D1) [node distance=2cm, below of=D] {${\varprojlim} {H}^{n}(F_{\alpha , \beta};G)$};
\node (E1) [node distance=2cm, below of=E] {0.};

\draw[->] (A) to node [above]{}(B);
\draw[->] (B) to node [above]{}(C);
\draw[->] (C) to node [above]{}(D);
\draw[->] (D) to node [above]{}(E);

\draw[->] (A1) to node [above]{}(B1);
\draw[->] (B1) to node [above]{}(C1);
\draw[->] (C1) to node [above]{}(D1);
\draw[->] (D1) to node [above]{}(E1);

\draw[->] (B) to node [right]{${\varprojlim}^1 ( {i_{\alpha , \beta }})^*$}(B1);
\draw[->] (C) to node [right]{$(i_{\alpha })^*$}(C1);
\draw[->] (D) to node [right]{${\varprojlim} ( {i_{\alpha , \beta }})^*$}(D1);
\end{tikzpicture}
\end{equation}
Therefore, the following diagram is commutative as well:
\begin{equation}
\begin{tikzpicture}
\node (A) {0};
\node (B) [node distance=3cm, right of=A] {${\varprojlim}^1 {H}^{n-1}(F_{\beta};G)$};
\node (C) [node distance=4cm, right of=B] {$H^n(X;G)$};
\node (D) [node distance=3.5cm, right of=C] {${\varprojlim} {H}^{n}(F_{\beta};G)$};
\node (E) [node distance=2.5cm, right of=D] {0};

\node (A1) [node distance=2cm, below of=A] {0};
\node (B1) [node distance=2cm, below of=B] {$\prod {\varprojlim}^1 {H}^{n-1}(F_{\alpha ,\beta};G)$};
\node (C1) [node distance=2cm, below of=C] {$\prod H^n(X_\alpha;G)$};
\node (D1) [node distance=2cm, below of=D] {$\prod {\varprojlim} {H}^{n}(F_{\alpha ,\beta};G)$};
\node (E1) [node distance=2cm, below of=E] {0.};

\draw[->] (A) to node [above]{}(B);
\draw[->] (B) to node [above]{}(C);
\draw[->] (C) to node [above]{}(D);
\draw[->] (D) to node [above]{}(E);

\draw[->] (A1) to node [above]{}(B1);
\draw[->] (B1) to node [above]{}(C1);
\draw[->] (C1) to node [above]{}(D1);
\draw[->] (D1) to node [above]{}(E1);

\draw[->] (B) to node [right]{$\prod {\varprojlim}^1 ( {i_{\alpha , \beta }})^*$}(B1);
\draw[->] (C) to node [right]{$\prod (i_{\alpha })^*$}(C1);
\draw[->] (D) to node [right]{$\prod {\varprojlim} (  {i_{\alpha , \beta }})^*$}(D1);
\end{tikzpicture}
\end{equation}
Note that for each $\beta$ a compact subspace $F_\beta $  has a nonempty intersection with $X_\alpha $ only for finitely   many $\alpha \in \mathcal{A}$ and therefore, by theorem 13.2c in \cite{3} we have the following isomorphism:
\begin{equation}
\begin{tikzpicture}
\node (A) {$H^n (F_\beta;G)$};
\node (B) [node distance=4cm, right of=A] {$\prod H^n(F_{\alpha , \beta} ;G).$};
\draw[->] (A) to node [above]{ $\cong $ }(B);
\end{tikzpicture}
\end{equation}
On the other hand, $\varprojlim \prod H^n(F_{\alpha , \beta} ;G) \cong  \prod \varprojlim H^n(F_{\alpha , \beta} ;G)$ and so the homomorphism 
\begin{equation}
\begin{tikzpicture}
\node (A) {${\varprojlim} {H}^{n}(F_{\beta};G)$};
\node (B) [node distance=6cm, right of=A] {$\prod {\varprojlim} {H}^{n}(F_{\alpha ,\beta};G)$};
\draw[->] (A) to node [above]{$\prod {\varprojlim} ( {i_{\alpha , \beta }})^*$ }(B);
\end{tikzpicture}
\end{equation}
is an isomorphism.

To complete the proof, it is sufficient to show that the homomorphism 
\begin{equation}
\begin{tikzpicture}
\node (A) {${\varprojlim}^{1} {H}^{*}(F_{\beta};G)$};
\node (B) [node distance=6cm, right of=A] {$\prod {\varprojlim}^{1} {H}^{*}(F_{\alpha ,\beta};G)$};
\draw[->] (A) to node [above]{$\prod {\varprojlim}^{1} ( {i_{\alpha , \beta }})^*$ }(B);
\end{tikzpicture}
\end{equation}
is an isomophism. 

Note that for each space $X \in \mathcal{K}^2_{\mathcal{P}ol_C}$, there exists a universal coefficients formula:
\begin{equation}
\begin{tikzpicture}
\node (A) {0};
\node (B) [node distance=2.5cm, right of=A] {$Ext(H_{n-1}(X),G)$};
\node (C) [node distance=3cm, right of=B] {$H^{n}(X;G)$};
\node (D) [node distance=3cm, right of=C] {$Hom(H_{n}(X),G)$};
\node (E) [node distance=2.5cm, right of=D] {0.};
\draw[->] (A) to node [above]{}(B);
\draw[->] (B) to node [above]{}(C);
\draw[->] (C) to node [above]{}(D);
\draw[->] (D) to node [above]{}(E);
\end{tikzpicture}
\end{equation}
Therefore, each $i_{\alpha, \beta}:F_{\alpha, \beta} \to F_{ \beta}$ inclusion induces the following commutative diagram:
\begin{equation}
\begin{tikzpicture}
\node (A) {0};
\node (B) [node distance=2.5cm, right of=A] {$Ext(H_{n-1}(F_{ \beta}),G)$};
\node (C) [node distance=3.5cm, right of=B] {$H^{n}(F_{ \beta};G)$};
\node (D) [node distance=3.5cm, right of=C] {$Hom(H_{n}(F_{ \beta}),G)$};
\node (E) [node distance=2.5cm, right of=D] {0};
\draw[->] (A) to node [above]{}(B);
\draw[->] (B) to node [above]{}(C);
\draw[->] (C) to node [above]{}(D);
\draw[->] (D) to node [above]{}(E);

\node (A1) [node distance=2cm, below of=A] {0};
\node (B1) [node distance=2.5cm, right of=A1] {$Ext(H_{n-1}(F_{\alpha, \beta}),G)$};
\node (C1) [node distance=3.5cm, right of=B1] {$H^{n}(F_{\alpha, \beta};G)$};
\node (D1) [node distance=3.5cm, right of=C1] {$Hom(H_{n}(F_{\alpha, \beta}),G)$};
\node (E1) [node distance=2.5cm, right of=D1] {0.};
\draw[->] (A1) to node [above]{}(B1);
\draw[->] (B1) to node [above]{}(C1);
\draw[->] (C1) to node [above]{}(D1);
\draw[->] (D1) to node [above]{}(E1);

\draw[->] (B) to node [right]{$Ext(  {i_{\alpha , \beta }})^*$}(B1);
\draw[->] (C) to node [right]{$(  {i_{\alpha , \beta }})^{*}$}(C1);
\draw[->] (D) to node [right]{$Hom(  {i_{\alpha , \beta }})^*$}(D1);
\end{tikzpicture}
\end{equation}
The groups $H_{*}(F_{\alpha, \beta})$ and $H_{*}(F_{ \beta})$ are finitely generated and so  by Corollary 1.5 of \cite{4} the diagram (3.7) induces the following  commutative diagram:
\begin{equation}
\begin{tikzpicture}
\node (A) {0};
\node (B) [node distance=3cm, right of=A] {$\varprojlim Ext(H_{n-1}(F_{ \beta}),G)$};
\node (C) [node distance=4.5cm, right of=B] {$\varprojlim H^{n}(F_{ \beta};G)$};
\node (D) [node distance=4.5cm, right of=C] {$\varprojlim Hom(H_{n}(F_{ \beta}),G)$};
\node (E) [node distance=3cm, right of=D] {0};
\draw[->] (A) to node [above]{}(B);
\draw[->] (B) to node [above]{}(C);
\draw[->] (C) to node [above]{}(D);
\draw[->] (D) to node [above]{}(E);

\node (A1) [node distance=2cm, below of=A] {0};
\node (B1) [node distance=3cm, right of=A1] {$\prod \varprojlim Ext(H_{n-1}(F_{\alpha, \beta}),G)$};
\node (C1) [node distance=4.5cm, right of=B1] {$\prod \varprojlim H^{n}(F_{\alpha, \beta};G)$};
\node (D1) [node distance=4.5cm, right of=C1] {$\prod \varprojlim Hom(H_{n}(F_{\alpha, \beta}),G)$};
\node (E1) [node distance=3cm, right of=D1] {0.};
\draw[->] (A1) to node [above]{}(B1);
\draw[->] (B1) to node [above]{}(C1);
\draw[->] (C1) to node [above]{}(D1);
\draw[->] (D1) to node [above]{}(E1);

\draw[->] (B) to node [right]{$\prod \varprojlim Ext(  {i_{\alpha , \beta }})^*$}(B1);
\draw[->] (C) to node [right]{$\prod \varprojlim (  {i_{\alpha , \beta }})^{*}$}(C1);
\draw[->] (D) to node [right]{$\prod \varprojlim Hom(  {i_{\alpha , \beta }})^*$}(D1);
\end{tikzpicture}
\end{equation}
On the other hand, we have:
\begin{equation}
\begin{tikzpicture}
\node (A) {$\varprojlim Hom(H_{n}(F_{ \beta}),G)$};
\node (B) [node distance=5cm, right of=A] {$ Hom( \varinjlim H_{n}(F_{\beta}),G)$};
\node (C) [node distance=5cm, right of=B] {$ Hom ( H^s_{n}(X),G)$};
\node (D) [node distance=3cm, right of=C] {};

\node (A1) [node distance=2cm, below of=A] {$ Hom (\sum H^s_{n}(X_\alpha ),G)$};
\node (A2) [node distance=2.5cm, left of=A1] {};
\node (B1) [node distance=2cm, below of=B] {$ Hom (\sum \varinjlim H_{n}(F_{\alpha , \beta}),G)$};
\node (C1) [node distance=2cm, below of=C] {$\prod Hom ( \varinjlim  H_{n}(F_{\alpha , \beta}),G)$};
\node (D1) [node distance=2cm, below of=D] {};

\node (A11) [node distance=2cm, below of=B1] {$\prod \varprojlim Hom (   H_{n}(F_{\alpha , \beta}),G).$};
\node (A21) [node distance=3.5cm, left of=A11] {};

\draw[->] (A) to node [above]{$\approx $ }(B);
\draw[->] (B) to node [above]{$\approx $ }(C);
\draw[->] (C) to node [above]{$\approx $ }(D);

\draw[->] (A2) to node [above]{$\approx $ }(A1);
\draw[->] (A1) to node [above]{$\approx $ }(B1);
\draw[->] (B1) to node [above]{$\approx $ }(C1);
\draw[->] (C1) to node [above]{$\approx $ }(D1);
\draw[->] (A21) to node [above]{$\approx $ }(A11);
\end{tikzpicture}
\end{equation}
Therefore, the homomorphism 
\begin{equation}
\begin{tikzpicture}
\node (A) {${\varprojlim} Hom( {H}_{n}(F_{\beta}),G)$};
\node (B) [node distance=8cm, right of=A] {$\prod \varprojlim Hom (   H_{n}(F_{\alpha , \beta}),G)$};
\draw[->] (A) to node [above]{$\prod {\varprojlim} Hom ( {i_{\alpha , \beta }})^*$ }(B);
\end{tikzpicture}
\end{equation}
is an isomorphism. By the isomorphisms (3.4), (3.10) and the diagram (3.8) we obtain that the homomorpism 
\begin{equation}
\begin{tikzpicture}
\node (A) {${\varprojlim} Ext ({H}_{n-1}(F_{\beta}),G)$};
\node (B) [node distance=8cm, right of=A] {$\prod \varprojlim  Ext (H_{n-1}(F_{\alpha , \beta}),G)$};
\draw[->] (A) to node [above]{$\prod {\varprojlim} Ext ( {i_{\alpha , \beta }})^*$ }(B);
\end{tikzpicture}
\end{equation}
is isomorphism as well. Note that ${H}_{n-1}(F_{\beta})$ and ${H}_{n-1}(F_{\alpha ,\beta})$ groups are finitely generated and so by proposition 1.2 and corollary 1.5 of \cite{4}, we have the following diagram:

\begin{equation}
\begin{tikzpicture}
\node (A) {0};
\node (B) [node distance=3cm, right of=A] {${\varprojlim}^1 Hom(H_{n-1}(F_{ \beta}),G)$};
\node (C) [node distance=6cm, right of=B] {$Ext (\varinjlim H_{n-1}(F_{ \beta}),G)$};
\node (F) [node distance=3cm, right of=C] {};

\node (D) [node distance=2cm, below of=C1] {$\varprojlim Ext(H_{n-1}(F_{ \beta}),G)$};
\node (N) [node distance=3.3cm, left of=D] {};
\node (E) [node distance=3cm, right of=D] {0};
\draw[->] (A) to node [above]{}(B);
\draw[->] (B) to node [above]{}(C);
\draw[->] (C) to node [above]{}(F);
\draw[->] (N) to node [above]{}(D);
\draw[->] (D) to node [above]{}(E);

\node (A1) [node distance=2cm, below of=A] {0};
\node (B1) [node distance=2cm, below of=B] {$\prod {\varprojlim}^1 Hom(H_{n-1}(F_{\alpha, \beta}),G)$};
\node (C1) [node distance=2cm, below of=C] {$\prod Ext (\varinjlim H_{n-1}(F_{\alpha, \beta}),G)$};
\node (F1) [node distance=3cm, right of=C1] {};
\node (D1) [node distance=2cm, below of=D] {$\prod \varprojlim Ext(H_{n-1}(F_{\alpha, \beta}),G)$};
\node (E1) [node distance=3cm, right of=D1] {0,};
\node (N1) [node distance=3.3cm, left of=D1] {};

\draw[->] (A1) to node [above]{}(B1);
\draw[->] (B1) to node [above]{}(C1);
\draw[->] (C1) to node [above]{}(F1);
\draw[->] (D1) to node [above]{}(E1);
\draw[->] (N1) to node [above]{}(D1);

\draw[->] (B) to node [right]{$\prod {\varprojlim}^1 Hom (  {i_{\alpha , \beta }})^*$}(B1);
\draw[->] (C) to node [right]{$\prod Ext \varinjlim (  {i_{\alpha , \beta }})^{*}$}(C1);
\draw[->] (D) to node [right]{$\prod \varprojlim Ext(  {i_{\alpha , \beta }})^*$}(D1);
\end{tikzpicture}
\end{equation}
where 
\begin{equation}
\begin{tikzpicture}
\node (A) {$Ext ({\varinjlim}  {H}_{n-1}(F_{\beta}),G)$};
\node (B) [node distance=8cm, right of=A] {$\prod Ext( {\varinjlim}  {H}_{n-1}(F_{\alpha ,\beta}),G)$};
\draw[->] (A) to node [above]{$\prod Ext~ {\varinjlim} ( {i_{\alpha , \beta }})^*$ }(B);
\end{tikzpicture}
\end{equation}
is an isomorphism. Indeed, 
\begin{equation}
\begin{tikzpicture}
\node (A) {$Ext ({\varinjlim} {H}_{n-1}(F_{\beta}),G)$};
\node (B) [node distance=4.5cm, right of=A] {$Ext ({H}^s_{n-1}(X),G)$};
\node (C) [node distance=5cm, right of=B] {$Ext (\sum {H}^s_{n-1}(X_\alpha),G)$};
\node (D) [node distance=3cm, right of=C] {};

\node (A1) [node distance=2cm, below of=A] {$\prod Ext ( {H}^s_{n-1}(X_\alpha),G)$};
\node (A2) [node distance=2.5cm, left of=A1] {};
\node (B1) [node distance=5.5cm, right of=A1] {$\prod Ext ( \varinjlim {H}_{n-1}(F_{\alpha , \beta}),G).$};
\node (D1) [node distance=2cm, below of=D] {};

\draw[->] (A) to node [above]{$\approx $ }(B);
\draw[->] (B) to node [above]{$\approx $ }(C);
\draw[->] (C) to node [above]{$\approx $ }(D);

\draw[->] (A2) to node [above]{$\approx $ }(A1);
\draw[->] (A1) to node [above]{$\approx $ }(B1);

\end{tikzpicture}
\end{equation}
Consequently, by the isomorphisms (3.11), (3.13) and the  commutative diagram (3.12) we obtain that the homomorphism 
\begin{equation}
\begin{tikzpicture}
\node (A) {${\varprojlim}^1 Hom(  {H}_{n-1}(F_{\beta}),G)$};
\node (B) [node distance=8cm, right of=A] {$\prod  {\varprojlim}^1 Hom(  {H}_{n-1}(F_{\alpha ,\beta}),G)$};
\draw[->] (A) to node [above]{$\prod  {\varprojlim}^1 Hom ( {i_{\alpha , \beta }})^*$ }(B);
\end{tikzpicture}
\end{equation}
is an isomorphism. 

While the groups $H_{n}(F_{ \beta};G)$ and $H_{n}(F_{\alpha , \beta};G)$ are finitely generated, by Corollary 1.5 of \cite{4}  we have the following commutative diagram:
\begin{equation}
\begin{tikzpicture}
\node (B) {};
\node (C) [node distance=2.5cm, right of=B] {${\varprojlim}^1 H^{n}(F_{ \beta};G)$};
\node (D) [node distance=5cm, right of=C] {${\varprojlim}^1 Hom(H_{n}(F_{ \beta}),G)$};
\node (E) [node distance=3cm, right of=D] {};
\draw[->] (C) to node [above]{$\sim$}(D);

\node (B1) [node distance=2cm, below of=A] {};
\node (C1) [node distance=2.5cm, right of=B1] {$\prod {\varprojlim}^1 H^{n}(F_{\alpha, \beta};G)$};
\node (D1) [node distance=5cm, right of=C1] {$\prod {\varprojlim}^1 Hom(H_{n}(F_{\alpha, \beta}),G).$};
\node (E1) [node distance=3cm, right of=D1] {};

\draw[->] (C1) to node [above]{$\sim$}(D1);

\draw[->] (C) to node [right]{$\prod {\varprojlim}^1 (  {i_{\alpha , \beta }})^{*}$}(C1);
\draw[->] (D) to node [right]{$\prod {\varprojlim}^1 Hom(  {i_{\alpha , \beta }})^*$}(D1);
\end{tikzpicture}
\end{equation}
By the isomorphism (3.15) and the diagram (3.16) the homomorphism 
\begin{equation}
\begin{tikzpicture}
\node (A) {${\varprojlim}^{1} {H}^{*}(F_{\beta};G)$};
\node (B) [node distance=6cm, right of=A] {$\prod {\varprojlim}^{1} {H}^{*}(F_{\alpha ,\beta};G)$};
\draw[->] (A) to node [above]{$\prod {\varprojlim}^{1} ( {i_{\alpha , \beta }})^*$ }(B);
\end{tikzpicture}
\end{equation}
is an isomophism as well. Therefore, by the isomorphisms (3.4), (3.17) and the diagram (3.2), finally we obtained that 
\begin{equation}
\begin{tikzpicture}
\node (A) {${H}^{n}(X;G)$};
\node (B) [node distance=5cm, right of=A] {$\prod {H}^{n}(X_{\alpha };G)$};
\draw[->] (A) to node [above]{$\prod ( {i_{\alpha }})^*$ }(B);
\end{tikzpicture}
\end{equation}
is an isomorphism.
\end{proof}

\begin{thm} If $H^*$ is a cohomology  theory in the Inasaridze-Mdzinarishvili-Beridze sense defined on the category $\mathcal{K}^2_\mathcal{CW}$,  then it is a cohomology theory in the Berikashvili-Mdzinarishvili-Beridze sense.
\end{thm} 
\begin{proof} For each space $X \in \mathcal{K}_\mathcal{CW}$ consider the direct system $ {\bf {F}} = \{ F_{\beta} \}_{\beta \in \mathcal{B}}$ of all compact subspaces. In this case, for each injective group $G_0$ we have an isomorphism:
\begin{equation}
 H^n(X;G_0) \approx \varprojlim H^n(F_\beta;G_0).
\end{equation}
Each $F_\beta$ has homotopy type of compact polyhedron, and so there is an exact sequence (the universal coefficients formula):
\begin{equation}
0 \to Ext(H_{n-1}(F_{\beta}),G_0) \to H^n (F_{\beta};G_0) \to Hom(H_n(F_{\beta}),G_0) \to 0,
\end{equation}
which induces the long exact sequence:
\begin{equation}
\begin{tikzpicture}
\node (A) {0};
\node (B) [node distance=3cm, right of=A] {$\varprojlim Ext(H_{n-1}(F_{\beta}),G_0)$};
\node (C) [node distance=4cm, right of=B] {$\varprojlim H^n (F_{\beta};G_0)$};
\node (D) [node distance=4cm, right of=C] {$\varprojlim Hom(H_n(F_{\beta}),G_0)$};
\node (E) [node distance=3cm, right of=D] {};

\draw[->] (A) to node [above]{}(B);
\draw[->] (B) to node [above]{}(C);
\draw[->] (C) to node [above]{}(D);
\draw[->] (D) to node [above]{}(E);

\node (A1) [node distance=2cm, below of=A] {};
\node (B1) [node distance=2cm, below of=B] {${\varprojlim}^1 Ext(H_{n-1}(F_{\beta}),G_0)$};
\node (C1) [node distance=2cm, below of=C] {${\varprojlim}^1 H^n (F_{\beta};G_0)$};
\node (D1) [node distance=2cm, below of=D] {${\varprojlim}^1 Hom(H_n(F_{\beta}),G_0)$};
\node (E1) [node distance=2cm, below of=E] {$\dots ~.$};

\draw[->] (A1) to node [above]{}(B1);
\draw[->] (B1) to node [above]{}(C1);
\draw[->] (C1) to node [above]{}(D1);
\draw[->] (D1) to node [above]{}(E1);

\end{tikzpicture}
\end{equation}
Note that for each injective group $G_0$ the functor $Ext(-,G_0)$ is trivial and by (3.21) we obtain the isomorphism:
\begin{equation}
\begin{tikzpicture}
\node (A) {${\varprojlim} {H}^{n}(F_\beta;G_0)$};
\node (B) [node distance=5cm, right of=A] {${\varprojlim} Hom ( {H}_{n}(F_\beta),G_0).$};
\draw[->] (A) to node [above]{$\approx$ }(B);
\end{tikzpicture}
\end{equation}
On the other hand, 
\begin{equation}
\begin{tikzpicture}
\node (A) {${\varprojlim} Hom ( {H}_{n}(F_\beta),G_0)$};
\node (B) [node distance=4.5cm, right of=A] {$Hom ({\varinjlim} {H}_{n}(F_\beta),G_0)$};
\node (C) [node distance=4.5cm, right of=B] {$Hom ({H}^s_{n}(X),G_0).$};
\draw[->] (A) to node [above]{$\approx$ }(B);
\draw[->] (B) to node [above]{$\approx$ }(C);
\end{tikzpicture}
\end{equation}
Therefore, for each injective group $G_0$ we have
\begin{equation}
\begin{tikzpicture}
\node (A) {${H}^{n}(X;G_0)$};
\node (B) [node distance=3.5cm, right of=A] {${\varprojlim}  H^n(F_\beta ;G_0)$};
\node (C) [node distance=4cm, right of=B] {$Hom (H^s_n (X) ,G_0).$};
\draw[->] (A) to node [above]{$\approx$ }(B);
\draw[->] (B) to node [above]{$\approx$ }(C);
\end{tikzpicture}
\end{equation}

Consider any abelian group $G$ and the corresponding injective resolution:
\begin{equation}
 ~~0 \to G \to G' \to G'' \to 0.
\end{equation}
Apply sequence (3.25) by the cohomological bifunctor $H^*(X;-)$, which gives the following long exact sequence:
\begin{equation} \dots  \to H^{n-1}(X;G') \to  H^{n-1}(X;G'') \to  H^n (X;G) \to H^{n}(X;G') \to  H^{n}(X;G'') \to \dots ,
\end{equation}
which induces the following exact sequence:
\begin{equation}
\begin{tikzpicture}

\node (A) {0};
\node (B) [node distance=4.5cm, right of=A] {$ Coker( H^{n-1} (X;G') \to  H^{n-1}(X;G''))$};
\node (C) [node distance=5cm, right of=B] {$ H^n (X;G)$};
\node (D) [node distance=2cm, right of=C] {};

\draw[->] (A) to node [above]{}(B);
\draw[->] (B) to node [above]{}(C);
\draw[->] (C) to node [above]{}(D);

\node (A1) [node distance=1.5cm, below of=A] {};
\node (B1) [node distance=1.5cm, below of=B] {$Ker( H^n (X;G') \to  H^n(X;G''))$};
\node (C1) [node distance=1.5cm, below of=C] {0.};

\draw[->] (A1) to node [above]{}(B1);
\draw[->] (B1) to node [above]{}(C1);

\end{tikzpicture}
\end{equation}
On the other hand, if we apply the functor $Hom (H^s_n(X),-)$ to the sequence (3.25) and take in to account that the groups $G'$ and $G''$ are injective, then we obtain the following exact sequence:

\begin{equation}
\begin{tikzpicture}

\node (A) {0};
\node (B) [node distance=2.5cm, right of=A] {$ Hom (H^s_n(X),G)$};
\node (C) [node distance=4cm, right of=B] {$ Hom ( H^s_n(X),G')$};
\node (D) [node distance=2.5cm, right of=C] {};

\draw[->] (A) to node [above]{}(B);
\draw[->] (B) to node [above]{}(C);
\draw[->] (C) to node [above]{}(D);

\node (A1) [node distance=1.5cm, below of=A] {};
\node (B1) [node distance=1.5cm, below of=B] {$ Hom (H^s_n(X),G'')$};
\node (C1) [node distance=1.5cm, below of=C] {$ Ext ( H^s_n(X),G)$};
\node (D1) [node distance=1.5cm, below of=D] {0.};

\draw[->] (A1) to node [above]{}(B1);
\draw[->] (B1) to node [above]{}(C1);
\draw[->] (C1) to node [above]{}(D1);

\end{tikzpicture}
\end{equation}
Therefore, for each integer $n \in \mathbb{Z}$ we have
\begin{equation}
 Ker( Hom( H^s_n(X),G') \to  Hom( H^s_n(X),G'')) \approx Hom(H^s_n(X),G),
\end{equation}
\begin{equation}
 Coker( Hom(H^s_n(X),G') \to  Hom( H^s_n(X),G'')) \approx Ext( H^s_n(X),G).
\end{equation}
Therefore, by (3.24), (3.29) and (3.30), the sequence (3.27) turned into the following sequence:
\begin{equation}
\begin{tikzpicture}

\node (A) {0};
\node (B) [node distance=2.5cm, right of=A] {$ Ext(H_{n-1}^s(X),G)$};
\node (C) [node distance=3cm, right of=B] {$ H^n (X;G)$};
\node (D) [node distance=3cm, right of=C] {$Hom(H_n^s(X),G)$};
\node (E) [node distance=2.5cm, right of=D] {0.};

\draw[->] (A) to node [above]{}(B);
\draw[->] (B) to node [above]{}(C);
\draw[->] (C) to node [above]{}(D);
\draw[->] (D) to node [above]{}(E);

\end{tikzpicture}
\end{equation}
\end{proof}

\begin{thm} If $H^*$ is a cohomology  theory in the Berikashvili-Mdzinarishvili-Beridze sense defined on the category $\mathcal{K}^2_\mathcal{CW}$,  then it is a cohomology theory in the Mdzinarishvili sense.
\end{thm} 

\begin{proof} For each space $X \in \mathcal{K}^2_\mathcal{CW}$ consider the direct system $ {\bf {F}} = \{ F_{\beta} \}_{\beta \in \mathcal{B}}$ of all compact subspaces. In this case, each space $F_\beta$ has a homotopy type of a compact  polyhedron and so there exists the following short exact sequence:
\begin{equation}
0 \to Ext({H}_{n-1}(F_{\beta}),G) \to H^n (F_{\beta};G) \to Hom({H}_n(F_{\beta}),G) \to 0,
\end{equation}
which induces the following long exact sequence:
\begin{equation}
\begin{tikzpicture}
\node (A) {0};
\node (B) [node distance=3cm, right of=A] {$\varprojlim Ext(H_{n-1}(F_{\beta}),G)$};
\node (C) [node distance=4cm, right of=B] {$\varprojlim H^n (F_{\beta};G)$};
\node (D) [node distance=4cm, right of=C] {$\varprojlim Hom(H_n(F_{\beta}),G)$};
\node (E) [node distance=3cm, right of=D] {};

\draw[->] (A) to node [above]{}(B);
\draw[->] (B) to node [above]{}(C);
\draw[->] (C) to node [above]{}(D);
\draw[->] (D) to node [above]{}(E);

\node (A1) [node distance=2cm, below of=A] {};
\node (B1) [node distance=2cm, below of=B] {${\varprojlim}^1 Ext(H_{n-1}(F_\beta),G)$};
\node (C1) [node distance=2cm, below of=C] {${\varprojlim}^1 H^n (F_\beta;G)$};
\node (D1) [node distance=2cm, below of=D] {${\varprojlim}^1 Hom(H_n(F_\beta),G)$};
\node (E1) [node distance=2cm, below of=E] {$\dots~.$};

\draw[->] (A1) to node [above]{}(B1);
\draw[->] (B1) to node [above]{}(C1);
\draw[->] (C1) to node [above]{}(D1);
\draw[->] (D1) to node [above]{}(E1);

\end{tikzpicture}
\end{equation}
By Corollary 1.5 in \cite{4} we have:
\begin{equation}
{\varprojlim}^r Ext({H}_{n-1}(F_{\beta}),G)=0, ~~r \geq 1. 
\end{equation}
Therefore, by (3.33) and (3.34) we obtain the exact sequence:
\begin{equation} 
0 \to \varprojlim Ext({H}_{n-1}(F_{\beta}),G) \to \varprojlim H^n (F_{\beta};G) \to \varprojlim Hom({H}_n(F_{\beta}),G) \to 0.
\end{equation}
Naturally, there exists a commutative diagram:

\begin{equation}
\begin{tikzpicture}
\node (A) {0};
\node (B) [node distance=3cm, right of=A] {$Ext(H^s_{n-1}(X),G)$};
\node (C) [node distance=4cm, right of=B] {$ H^n (X;G)$};
\node (D) [node distance=4cm, right of=C] {$ Hom(H^s_n(X),G)$};
\node (E) [node distance=3cm, right of=D] {0};

\draw[->] (A) to node [above]{}(B);
\draw[->] (B) to node [above]{}(C);
\draw[->] (C) to node [above]{}(D);
\draw[->] (D) to node [above]{}(E);

\node (A1) [node distance=2cm, below of=A] {0};
\node (B1) [node distance=2cm, below of=B] {${\varprojlim} Ext(H_{n-1}(F_{\beta}),G)$};
\node (C1) [node distance=2cm, below of=C] {${\varprojlim} H^n (F_{\beta};G)$};
\node (D1) [node distance=2cm, below of=D] {${\varprojlim} Hom(H_n(F_{\beta}),G)$};
\node (E1) [node distance=2cm, below of=E] {0.};

\draw[->] (A1) to node [above]{}(B1);
\draw[->] (B1) to node [above]{}(C1);
\draw[->] (C1) to node [above]{}(D1);
\draw[->] (D1) to node [above]{}(E1);

\draw[->] (B) to node [right]{$\psi$}(B1);
\draw[->] (C) to node [right]{$\varphi$}(C1);
\draw[->] (D) to node [right]{$\chi$}(D1);

\end{tikzpicture}
\end{equation}
On the other hand by (3.23), we have $ \chi : Hom(H^s_n(X),G)=Hom( \varinjlim H_n(F_\beta ),G) \to \varprojlim Hom({H}_n(F_{\beta}),G)$ is the isomorphism and so
\begin{equation}
Ker \psi \approx Ker \varphi,    ~~~ Coker \psi \approx Coker \varphi.
\end{equation}
Therefore, we obtain the following commutative diagram of the exact sequences:

\begin{equation}
\begin{tikzpicture}
\node (A) {0};
\node (B) [node distance=2cm, right of=A] {$Ker \psi$};
\node (C) [node distance=3cm, right of=B] {$Ext(H^s_{n-1}(X),G)$};
\node (D) [node distance=4.2cm, right of=C] {$\varprojlim Ext (H_{n-1} (F_\beta),G)$};
\node (E) [node distance=3.2cm, right of=D] {$ Coker \psi$};
\node (F) [node distance=2cm, right of=E] {0};

\draw[->] (A) to node [above]{}(B);
\draw[->] (B) to node [above]{}(C);
\draw[->] (C) to node [above]{$\psi$}(D);
\draw[->] (D) to node [above]{}(E);
\draw[->] (E) to node [above]{}(F);

\node (A1) [node distance=2cm, below of=A] {0};
\node (B1) [node distance=2cm, below of=B] {$Ker \varphi$};
\node (C1) [node distance=2cm, below of=C] {$H^{n} (X;G)$};
\node (D1) [node distance=2cm, below of=D] {${\varprojlim} H^{n}(F_{\beta};G)$};
\node (E1) [node distance=2cm, below of=E] {$Coker \varphi$};
\node (F1) [node distance=2cm, below of=F] {0.};

\draw[->] (A1) to node [above]{}(B1);
\draw[->] (B1) to node [above]{}(C1);
\draw[->] (C1) to node [above]{$\varphi$}(D1);
\draw[->] (D1) to node [above]{}(E1);
\draw[->] (E1) to node [above]{}(F1);

\draw[->] (B) to node [right]{$\approx$}(B1);
\draw[->] (C) to node [right]{}(C1);
\draw[->] (D) to node [right]{}(D1);
\draw[->] (E) to node [right]{$\approx$}(E1);

\end{tikzpicture}
\end{equation}
On the other hand, $H^s_{n-1}(X) \approx \varinjlim H_{n-1}(F_\beta)$ and so we obtain:

\begin{equation}
\begin{tikzpicture}
\node (A) {0};
\node (B) [node distance=2cm, right of=A] {$Ker \psi$};
\node (C) [node distance=3cm, right of=B] {$Ext(\varinjlim H_{n-1}(F_\beta),G)$};
\node (D) [node distance=4.5cm, right of=C] {$\varprojlim Ext (H_{n-1} (F_\beta),G)$};
\node (E) [node distance=3.2cm, right of=D] {$ Coker \psi$};
\node (F) [node distance=2cm, right of=E] {0};

\draw[->] (A) to node [above]{}(B);
\draw[->] (B) to node [above]{}(C);
\draw[->] (C) to node [above]{$\psi$}(D);
\draw[->] (D) to node [above]{}(E);
\draw[->] (E) to node [above]{}(F);

\node (A1) [node distance=2cm, below of=A] {0};
\node (B1) [node distance=2cm, below of=B] {$Ker \varphi$};
\node (C1) [node distance=2cm, below of=C] {$H^{n} (X;G)$};
\node (D1) [node distance=2cm, below of=D] {${\varprojlim} H^{n}(F_{\beta};G)$};
\node (E1) [node distance=2cm, below of=E] {$Coker \varphi$};
\node (F1) [node distance=2cm, below of=F] {0.};

\draw[->] (A1) to node [above]{}(B1);
\draw[->] (B1) to node [above]{}(C1);
\draw[->] (C1) to node [above]{$\varphi$}(D1);
\draw[->] (D1) to node [above]{}(E1);
\draw[->] (E1) to node [above]{}(F1);

\draw[->] (B) to node [right]{$\approx$}(B1);
\draw[->] (C) to node [right]{}(C1);
\draw[->] (D) to node [right]{}(D1);
\draw[->] (E) to node [right]{$\approx$}(E1);

\end{tikzpicture}
\end{equation}
By Proposition 1.2 in \cite{4} for the direct system $\{ H_{n-1}(F_\beta) \}_{\beta \in \mathcal{B}}$ of the homological groups we have:
\begin{equation}
\begin{tikzpicture}
\node (A) {0};
\node (B) [node distance=3.5cm, right of=A] {${\varprojlim}^1 Hom (H_{n-1}(F_\beta),G)$};
\node (C) [node distance=5.5cm, right of=B] {$ Ext ({\varinjlim} H_{n-1}(F_\beta),G)$};
\node (D) [node distance=3.5cm, right of=C] {};

\node (B1) [node distance=2cm, below of=B] {${\varprojlim} Ext (H_{n-1}(F_\beta),G)$};
\node (C1) [node distance=2cm, below of=C] {${\varprojlim}^2 Hom (H_{n-1}(F_\beta),G)$};
\node (D1) [node distance=2cm, below of=D] {0.};
\node (A1) [node distance=2cm, below of=A] {};

\draw[->] (A) to node [above]{}(B);
\draw[->] (B) to node [above]{$\psi$}(C);
\draw[->] (C) to node [above]{}(D);

\draw[->] (A1) to node [above]{}(B1);
\draw[->] (B1) to node [above]{}(C1);
\draw[->] (C1) to node [above]{}(D1);

\end{tikzpicture}
\end{equation}
Consequently, by (3.39) we obtain the exact sequence:
\begin{equation}
\begin{tikzpicture}
\node (A) {0};
\node (B) [node distance=3.5cm, right of=A] {${\varprojlim}^1 Hom (H_{n-1}(F_\beta),G)$};
\node (C) [node distance=5cm, right of=B] {$  H^n(X;G)$};
\node (D) [node distance=3cm, right of=C] {};

\node (B1) [node distance=2cm, below of=B] {${\varprojlim} H^{n}(F_\beta;G)$};
\node (C1) [node distance=2cm, below of=C] {${\varprojlim}^2 Hom (H_{n-1}(F_\beta),G)$};
\node (D1) [node distance=2cm, below of=D] {0.};
\node (A1) [node distance=2cm, below of=A] {};

\draw[->] (A) to node [above]{}(B);
\draw[->] (B) to node [above]{}(C);
\draw[->] (C) to node [above]{}(D);

\draw[->] (A1) to node [above]{}(B1);
\draw[->] (B1) to node [above]{}(C1);
\draw[->] (C1) to node [above]{}(D1);

\end{tikzpicture}
\end{equation}
The homology groups $H_{n-1}(F_\beta)$ are finitely generated and by Corollary 1.5 (see 2e, 3e) \cite{4} we have:
\begin{equation}
{\varprojlim}^2 Hom(H_{n-1}(F_\beta),G)=0.
\end{equation}
Using the the long exact sequence (3.33) and equality (3.34), we have
\begin{equation}
{\varprojlim}^1 Hom(H_{n-1}(F_\beta),G)={\varprojlim}^1 H^{n-1}(F_\beta; G).
\end{equation}
Consequently, by (3.41) we obtain the exact sequence:
\begin{equation}
\begin{tikzpicture}
\node (A) {0};
\node (B) [node distance=2.5cm, right of=A] {${\varprojlim}^1 H^{n-1}(F_\beta;G)$};
\node (C) [node distance=3.5cm, right of=B] {$  H^n(X;G)$};
\node (D) [node distance=3cm, right of=C] {${\varprojlim} H^{n}(F_\beta;G)$};
\node (E) [node distance=2cm, right of=D] {0.};

\draw[->] (A) to node [above]{}(B);
\draw[->] (B) to node [above]{}(C);
\draw[->] (C) to node [above]{}(D);
\draw[->] (D) to node [above]{}(E);

\end{tikzpicture}
\end{equation}

\end{proof}

\begin{thm} If $H^*$ is a cohomology  theory  in the Mdzinarishvili sense defined on the category $\mathcal{K}^2_\mathcal{CW}$,  then it is a cohomology theory in the Inasaridze-Mdzinarishvili-Beridze sense.
\end{thm} 

\begin{proof} By Theorem 3.1,  $H^*$ is a cohomology theory in the Milnor sense and by theorem 1.2, it is natural isomorphic to the singular $H^*_s$ cohomology and so, it is an exact functor of the second argument. Therefore, to prove the theorem it is sufficient to show that it has compact support for injective coefficients group.
	
For each space $X \in \mathcal{K}^2_\mathcal{CW}$ consider a direct system $ {\bf {F}} = \{ F_{\beta} \}_{\beta \in \mathcal{B}}$ of all compact subspaces. Then, by condition of the theorem, we have the following short exact sequence:
\begin{equation}
\begin{tikzpicture}
\node (A) {0};
\node (B) [node distance=2.5cm, right of=A] {${\varprojlim}^1 H^{n-1}(F_\beta;G)$};
\node (C) [node distance=3.5cm, right of=B] {$  H^n(X;G)$};
\node (D) [node distance=3cm, right of=C] {${\varprojlim} H^{n}(F_\beta;G)$};
\node (E) [node distance=2cm, right of=D] {0.};

\draw[->] (A) to node [above]{}(B);
\draw[->] (B) to node [above]{}(C);
\draw[->] (C) to node [above]{}(D);
\draw[->] (D) to node [above]{}(E);

\end{tikzpicture}
\end{equation}
 Therefore, to prove the theorem, it is sufficient  to show that for each injective group $G$ the first derivative group ${\varprojlim}^1 H_{n-1}(F_\beta; G)$ is trivial. Indeed, by (3.33) and (3.34) we have the following isomorphism:
\begin{equation}
 {\varprojlim}^1 H^n (F_\beta;G) \approx {\varprojlim}^1 Hom(H_n(F_\beta),G).
 \end{equation}
On the other hand, for each direct system $ \{ H_n(F_\beta) \}_{\beta \in \mathcal{B}}$ of abelian groups we have:
\begin{equation}
\begin{tikzpicture}
\node (A) {0};
\node (B) [node distance=3.5cm, right of=A] {${\varprojlim}^1 Hom (H_{n}(F_\beta),G)$};
\node (C) [node distance=5.5cm, right of=B] {$ Ext ({\varinjlim} H_{n}(F_\beta),G)$};
\node (D) [node distance=3.5cm, right of=C] {};

\node (B1) [node distance=2cm, below of=B] {${\varprojlim} Ext (H_{n}(F_\beta),G)$};
\node (C1) [node distance=2cm, below of=C] {${\varprojlim}^2 Hom (H_{n}(F_\beta),G)$};
\node (D1) [node distance=2cm, below of=D] {0.};
\node (A1) [node distance=2cm, below of=A] {};

\draw[->] (A) to node [above]{}(B);
\draw[->] (B) to node [above]{$\psi$}(C);
\draw[->] (C) to node [above]{}(D);

\draw[->] (A1) to node [above]{}(B1);
\draw[->] (B1) to node [above]{}(C1);
\draw[->] (C1) to node [above]{}(D1);

\end{tikzpicture}
\end{equation}
For an injective group $G$ the functor $Ext(-,G)$ is trivial and so by (3.47) we have
\begin{equation}
 ~~ {\varprojlim}^1 Hom(H_{n}(F_\beta),G) \approx 0.
\end{equation}

Therefore, by (3.43), (3.45), (3.46) and (3.48) for each injective group $G$ we have
\begin{equation}
H^n(X;G) \approx \varprojlim H^n (F_ \beta ;G).
\end{equation} 

\end{proof}
By theorems 1.1, 1.6, 2.1, 3.1, 3.2, 3.3 and 3.4 we have: 

\begin{cor} Let $H^*$ be a cohomology theory on the category $\mathcal{K}^2_{\mathcal{CW}}$ with coefficients group $G$ in the Milnor or  Mdzinarishvili or Berikashvili-Mdzinarishvili-Beridze or Inasaridze-Mdzinarishvili-Beridze sense, then for each $(X,A)$ in $\mathcal{K}^2_{\mathcal{CW}}$ there is a natural isomorphism between $H^n(X,A;G)$ and the n-th singular cohomology group $H^n_s(X,A;G)$ of $(X,A)$ with coefficients in $G$. 
\end{cor}

\subsection*{Acknowledgements}
\
The first author was supported by the institutional scientific research project  of Batumi Shota Rustaveli State University.


\begin{thebibliography}{HD}




\normalsize
\baselineskip=17pt




\bibitem[Ber]{1} Berikashvili, N. A. Axiomatics of the Steenrod-Sitnikov homology theory on the category of compact Hausdorff spaces. (Russian) Topology (Moscow, 1979). Trudy Mat. Inst. Steklov. 154 (1983), 24--37


\bibitem[Ber-Mz]{2} Beridze, Anzor; Mdzinarishvili, Leonard. On the axiomatic systems of Steenrod homology theory of compact spaces. Topology Appl. 249 (2018), 73--82

\bibitem[E-S]{3} Eilenberg, Samuel; Steenrod, Norman. Foundations of algebraic topology. Princeton University Press, Princeton, New Jersey, 1952


\bibitem[H-M]{4} Huber, Martin; Meier, Willi. Cohomology theories and infinite $CW$-complexes. Comment. Math. Helv. 53 (1978), no. 2, 239--257.

\bibitem[In]{9} Inassaridze, Hvedri. On the Steenrod homology theory of compact spaces. Michigan Math. J. 38 (1991), no. 3, 323--338

\bibitem[In-Mdz]{5} Inasaridze, Kh. N.; Mdzinarishvili, L. D. On the connection between continuity and exactness in homology theory. (Russian) Soobshch. Akad. Nauk Gruzin. SSR 99 (1980), no. 2, 317--320. 

\bibitem[Mdz$_1$]{6} Mdzinarishvili, L. The uniqueness theorem for cohomologies on the category of polyhedral pairs. Trans. A. Razmadze Math. Inst. 172 (2018), no. 2, 265--275. 

\bibitem[Mdz$_2$]{7}Mdzinarishvili, L. D. On homology extensions. Glas. Mat. Ser. III 21(41) (1986), no. 2, 455--482


\bibitem[Mil]{8} Milnor, J. On axiomatic homology theory. Pacific J. Math. 12 1962 337--341 





\end{thebibliography}
\end{document}